\documentclass[11pt]{amsart}
\usepackage{amssymb,pb-diagram}
\usepackage{amsmath}
\usepackage{amsthm,amssymb,amsfonts}
\usepackage{amssymb,amscd}
\usepackage{epic,eepic,epsfig}
\usepackage{latexsym}
\usepackage{xcolor}
\usepackage{xy}
\xyoption{all}

\newtheorem{Thm}{Theorem}[section]

\newtheorem{Cor}[Thm]{Corollary}
\newtheorem{Lem}[Thm]{Lemma}
\newtheorem{Prop}[Thm]{Proposition}
\newtheorem{Claim}[Thm]{Claim}

\theoremstyle{definition}
\newtheorem{Def}[Thm]{Definition}
\newtheorem{Exa}[Thm]{Example}
\newtheorem{Remark}[Thm]{Remark}

\theoremstyle{remark}

\def\ldots{\mathinner{\ldotp\ldotp\ldotp}}

\def \cal{\mathcal}
\def \Bbb{\mathbb}

\def\eps{\varepsilon}

\def\Ndb{\mathbb N}

\def\Rdb{\mathbb R}

\newcommand{\wsconv}{\overline{\rm conv}^*}

\newcommand{\Natural}{\mathbb N}

\newcommand{\Real}{\mathbb R}
\newcommand{\abs}[1]{\left\vert#1\right\vert}
\newcommand{\set}[1]{\left\{#1\right\}}
\newcommand{\restricted}{\upharpoonright}

\newcommand{\dist}{\mathop{\mathrm{dist}}\nolimits}

\def\newball{B_{|\ |}}
\def\wstoo{\stackrel{w^*}{\longrightarrow}}

\newcommand{\bib}{\bibitem}

\begin{document}

\title{Szlenk indices of convex hulls}

\author{G. Lancien$^\clubsuit$}
\address{$\clubsuit$ Laboratoire de Math\'ematiques de Besan\c con, Universit\'e Bourgogne  Franche-Comt\'e, 16 route de Gray, 25030 Besan\c con C\'edex, France}
\email{gilles.lancien@univ-fcomte.fr}

\author{A. Proch\'{a}zka$^\diamondsuit$}
\address{$\diamondsuit$ Laboratoire de Math\'ematiques de Besan\c con, Universit\'e Bourgogne  Franche-Comt\'e, 16 route de Gray, 25030 Besan\c con C\'edex, France}
\email{antonin.prochazka@univ-fcomte.fr}

\author{M. Raja$^\spadesuit$}
\address{$\spadesuit$ Departamento de Matem\'{a}ticas, Universidad de Murcia, Campus de Espinardo, 30100 Espinardo, Murcia, Spain}
\email{matias@um.es}


\thanks{The second named author was partially supported by a ``Bonus Qualit\'{e} Recherche'' from the Universit\'{e} de Franche-Comt\'{e}.}
\thanks{The third named author was partially supported by the grants MINECO/FEDER MTM2014-57838-C2-1-P and Fundaci\'on S\'eneca CARM 19368/PI/14.}


\maketitle

\begin{abstract} We study the general measures of non-compactness defined on subsets of a dual Banach space, their associated derivations and their $\omega$-iterates. We introduce the notions of convexifiable and sublinear measure of non-compactness and investigate the properties of its associated fragment and slice derivations. We apply our results to the Kuratowski measure of non-compactness and to the study of the Szlenk index of a Banach space. As a consequence, we obtain that the Szlenk index and the convex Szlenk index of a separable Banach space are always equal. We also give, for any countable ordinal $\alpha$, a characterization of the Banach spaces with Szlenk index bounded by $\omega^{\alpha+1}$ in terms of the existence of an equivalent renorming. This extends a result by Knaust, Odell and Schlumprecht on Banach spaces with Szlenk index equal to $\omega$.
\end{abstract}

\section{Introduction}

In this paper we deal with the Szlenk and the convex Szlenk index of Banach spaces. Let us first recall their definitions.
Let $X$ be a Banach space, $K$ a weak$^*$-compact subset of its dual $X^*$ and $\eps>0$.
Then we define
$$s_\eps'(K)=\{ x^* \in K,\ {\rm for\ any}\ {\rm weak^*-neighborhood}\ U\ {\rm of}\ x^*,\  {\rm diam}(K \cap U) \ge \varepsilon\}$$
and inductively the sets $s_\eps^\alpha(K)$ for $\alpha$ ordinal as follows: $s_\eps^{\alpha+1}(K)=s_\eps'(s_\eps^\alpha(K))$ and $s_\eps^\alpha(K)=\bigcap_{\beta<\alpha}s_\eps^\beta(K)$ if $\alpha$ is a limit ordinal.\\
Then $Sz(K,\eps)=\inf\{\alpha,\ s^\alpha_\eps(K)=\emptyset\}$ if it exists and we denote $Sz(K,\eps)=\infty$ otherwise. Next we define $Sz(K)=\sup_{\eps>0}Sz(K,\eps)$.
The closed unit ball of $X^*$ is denoted $B_{X^*}$ and the Szlenk index of $X$ is $Sz(X)=Sz(B_{X^*})$.\\
Let us also denote $A^0_\varepsilon:=K$,  $A^{\alpha+1}_\eps$ is the weak$^*$-closed convex hull of $s'_\eps(A^\alpha_\eps)$ and $A^{\alpha}_\eps:=\bigcap_{\beta<\alpha}A^{\beta}_\eps$ if $\alpha$ is a limit ordinal.
Now we set $Cz(K,\eps)=\inf\{\alpha,\ A^{\alpha}_\eps=\emptyset\}$ if it exists and $Cz(K,\eps)=\infty$ otherwise. Then $Cz(K)=\sup_{\eps>0}Cz(K,\eps)$ and the convex Szlenk index of $X$ is  $Cz(X)=Cz(B_{X^*})$.

The Szlenk index  was first introduced by W. Szlenk \cite{Szlenk1968}, in a slightly different form, in order to prove that there is no separable reflexive Banach space universal for the class of all separable reflexive Banach spaces.
Another striking fact is that the isomorphic classification of a separable $C(K)$ space is perfectly determined by the value of its Szlenk index.
This is a consequence of some classical work by C.~Bessaga and  A.~Pe\l czy\'nski \cite{BessagaPelczynski1960}, D.E.~Alspach and Y.~Benyamini \cite{AlspachBenyamini1979},  C.~Samuel \cite{Samuel1983} and A.A.~Milutin \cite{Milutin1966} (see also \cite{Rosenthal2003} for a survey on $C(K)$-spaces).

One of the goals of this paper is to obtain a general renorming result for Banach spaces with a prescribed Szlenk index in the spirit of the following important theorem due to Knaust, Odell and Schlumprecht \cite{KnaustOdellSchlumprecht1999}: if $X$ is a separable Banach space and $Sz(X)\le \omega$, where $\omega$ is the first infinite ordinal, then $X$ admits an equivalent norm whose dual norm is such that for any $\eps>0$, there exists $\delta >0$ satisfying $s_\eps'(B)\subset (1-\delta)B$, where $B$ is the closed unit ball of this equivalent dual norm. Such a norm is said to be weak$^*$-uniformly Kadets-Klee. This has been quantitatively improved in \cite{GKL2001} and extended to the non-separable case in \cite{Raja2010}. We will prove a similar result for other values of the Szlenk index. More precisely, we show (Theorem \ref{renorming}) that a separable Banach space $X$ satisfies  $Sz(X)\le \omega^{\alpha+1}$ with $\alpha$ countable if and only if $X$ admits an equivalent norm whose dual norm is such that for any $\eps>0$, there exists $\delta >0$ satisfying $s_\eps^{\omega^\alpha}(B)\subset (1-\delta)B$, where $B$ is the closed unit ball of this equivalent dual norm. We say that such a norm is $\omega^\alpha$-UKK$^*$. It is worth recalling that $Sz(X)$ is always of the form $\omega^\alpha$ (see Lemma~\ref{formSz} for a slightly more general statement, \cite{Sersouri1989} for the original idea, or \cite{Lancien2006}). Let us also mention that C.~Samuel proved in \cite{Samuel1983} that $Sz(C_0([0,\omega^{\omega^\alpha})))=\omega^{\alpha+1}$ whenever $\alpha$ is a countable ordinal. A different proof of this computation is given in \cite{HajekLancien2007} by showing that the natural norm of $C_0([0,\omega^{\omega^\alpha}))$ is $\omega^\alpha$-UKK$^*$.

One of the inconveniences of the Szlenk derivation is that it does not preserve convexity. This explains why it is difficult to obtain renorming results related to this derivation. In contrast, a derivation based on peeling off slices (i.e. intersections with half spaces) preserves the convexity and allows to use distance functions to the derived sets in order to build a good equivalent norm (see \cite{Lancien1995} for instance).

In order to overcome this difficulty, we will study the fragment and slice derivations associated with general measures of non-compactness as they were introduced in \cite{Raja2007}. In section \ref{noncompactness}, we recall these definitions and also introduce the $\omega$-iterated measure $\eta^\omega$ associated with a measure of non-compactness $\eta$. In section \ref{convexifiable}, we introduce the notions of convexifiable and homogeneous measure of non-compactness. Then we prove a crucial result (Proposition \ref{convex}) on the properties of the slice derivation associated with a general convexifiable measure of non-compactness. In section \ref{sublinear} we explore the notion of sublinear measure of non-compactness and obtain a sharp comparison between the slice and fragment derivations for certain convexifiable sublinear measures of non-compactness. Section \ref{kuratowski} is devoted to the applications of our general results to the Kuratowski measure of non-compactness, denoted $\sigma$. This measure is of special interest to us, as its fragment derivation is exactly the Szlenk derivation. The main result of this section is Theorem \ref{main} which asserts that for any $n\in \Ndb$, the iterate $\sigma^{\omega^n}$ of $\sigma$ is convexifiable. It is then a key ingredient for proving that $Cz(K)=Sz(K)$ for any weak$^*$-compact subset $K$ of the dual of a separable Banach space such that $Sz(K)\le \omega^n$, for some integer $n$ (Corollary \ref{SzCz}). This together with some earlier work by P. H\'{a}jek and T. Schlumprecht \cite{HajekSchlumprecht2014} implies that for any separable Banach space $X$, $Cz(X)=Sz(X)$.
Finally, with all these ingredients, we state and prove our renorming result in section \ref{norm}.

\section{Measures of non-compactness and associated derivations}\label{noncompactness}

\begin{Def}\label{def1} Let $X$ be a Banach space. We call {\it measure of non-compactness on} $X^*$, any map $\eta$ defined on the weak$^*$-compact subsets of $X^*$ with values in $[0,\infty)$ and satisfying the following properties:
\begin{itemize}
\item[(i)] $\eta(\{x^*\})=0$ for any $x^* \in X^*$.
\item[(ii)] If $A_1,..,A_n$ are weak$^*$-compact subsets of $X^*$, then $\eta(\bigcup_{i=1}^n A_i)=\max_{i} \eta(A_i)$.
\item[(iii)] There exists $b>0$ such that for any weak$^*$-compact subset $A$ of $X^*$ and any $\lambda> 0$, $\eta(A + \lambda B_{X^*}) \leq \eta(A) + \lambda b  $.
\end{itemize}
\end{Def}

\begin{Remark} Note that (ii) implies that a measure of non-compactness $\eta$ on $X^*$ is monotone in the following sense: $\eta(A) \leq \eta(B)$, whenever $A \subset B$ are weak$^*$-compact subsets of $X^*$.\\
Note also that if $\eta$ is a measure of non-compactness on $X^*$, then $\eta(A)=0$, for any norm-compact subset $A$ of $X^*$.
\end{Remark}

\begin{Exa} Let $A$ be a weak$^*$-compact subset of $X^*$. We call {\it Kuratowski measure of non-compactness of} $A$ and denote $\sigma(A)$, the infimum of all $\eps>0$ such that $A$ can be covered by a finite family of closed balls of diameter equal to $\eps$.\\
Conditions (i), (ii) and (iii) are clearly satisfied. Note that (iii) is satisfied with $b=1$.
\end{Exa}

\medskip Following \cite{Raja2007}, we now define two set operations associated with a given measure of non-compactness $\eta$. Let us agree that if $E$ is a subset of a dual Banach space $X$, then $\overline{E}^*$ denotes the weak$^*$-closure of $E$.

We start with the {\it fragment derivation}. Given $\eps>0$ and $A$ a weak$^*$-compact subset of $X^*$, we set
$$ [\eta]'_{\varepsilon}(A)=\{ x^* \in A,\ {\rm for\ any}\ {\rm weak^*-neighborhood}\ U\ {\rm of}\ x^*,\  \eta(A \cap \overline{U}^*) \ge \varepsilon\}.$$
For any ordinal $\gamma$, the sets $[\eta]^\gamma_{\varepsilon}(A)$ are defined by $[\eta]^{\gamma+1}_{\varepsilon}(A)=[\eta]'_{\varepsilon}([\eta]^\gamma_{\varepsilon}(A))$ and $[\eta]^\gamma_{\varepsilon}(A)=\bigcap_{\beta<\gamma}[\eta]^\beta_{\varepsilon}(A)$ if $\gamma$ is a limit ordinal.\\

Similarly we define the {\it slice derivation} by
$$ \langle \eta\rangle '_{\varepsilon}(A)=\{ x^* \in A,\ {\rm for\ any}\ {\rm weak^*-open\ halfspace}\ H\ {\rm containing}\ x^*,\   \eta(A \cap \overline{H}^*) \ge\varepsilon\}.$$
Then, for an ordinal $\gamma$, the  set $\langle \eta\rangle^\gamma_{\varepsilon}(A)$ is defined in an obvious way as before. Note that we clearly have that $[\eta]^\gamma_{\varepsilon}(A) \subset \langle\eta\rangle^\gamma_{\varepsilon}(A)$.

\medskip We begin with a very basic property of the fragment derivation associated with a measure of non-compactness.

\begin{Lem}\label{referee1} Let $X$ be a Banach space and $\eta$ a measure of non-compactness on $X^*$ and let $A$ be a weak$^*$-compact subset of $X^*$. Then for any  $B$ weak$^*$-closed subset of $A$ and any $\eps>0$ such that $B\cap [\eta]'_\eps(A)=\emptyset$, we have that $\eta(B)<\eps$.
\end{Lem}

\begin{proof} For any $x^*$ in $B$ there exists a weak$^*$-neighborhood $U_{x^*}$ of $x^*$ so that $\eta(A \cap \overline{U_{x^*}}^*)<\eps$. Then, since $B$ is included in $A$ and weak$^*$-compact, there exist $x_1^*,..,x_n^*$ in $B$ such that $B\subset \bigcup_{k=1}^n (A \cap \overline{U_{x_k^*}}^*).$\\
The conclusion now follows from property (ii) of Definition \ref{def1}.
\end{proof}

We now define the iterates of a measure of non-compactness.

\begin{Def} Let $X$ be a Banach space and $\eta$ a measure of non-compactness on $X^*$. The $\omega$-iterated measure of $\eta$ is defined by
$$ \eta^{\omega}(A)=\inf\{ \varepsilon>0: [\eta]^{\omega}_{\varepsilon}(A) = \emptyset \}. $$
It will also be convenient to define $\eta^1=\eta^{\omega^0}:=\eta$ and then inductively $\eta^{\omega^{n+1}}:=(\eta^{\omega^n})^\omega$.
\end{Def}

\begin{Lem}\label{iterate} If $\eta$ is a measure of non-compactness on $X^*$, then
$\eta^{\omega}$ is a measure of non-compactness on $X^*$.
\end{Lem}

\begin{proof} Condition (i) of the definition is clearly satisfied by  $\eta^{\omega}$.\\
The condition (ii) for $\eta$ yields that $[\eta]'_{\varepsilon}(\bigcup_{i=1}^n  A_i) = \bigcup_{i=1}^n  [\eta]'_{\varepsilon}(A_i)$ whenever $A_1,..,A_n$ are weak$^*$-compact subsets of $X^*$. After iterating, this implies that $\eta^{\omega}$ satisfies condition (ii) of the definition.\\
Property (iii) comes from the following observation. Let $b>0$ be the constant given by condition (iii) for $\eta$. Then for any weak$^*$-compact subset $A$ of $X^*$ and any $\lambda> 0$,
\begin{equation}\label{eq1}
[\eta]'_{\varepsilon+\lambda b}(A+\lambda B_{X^*}) \subset [\eta]'_{\varepsilon}(A) + \lambda B_{X^*}.
\end{equation}
Indeed, take any
$$x^* \in A+\lambda B_{X^*} \setminus ([\eta]'_{\varepsilon}(A) + \lambda B_{X^*}).$$
Then, there exists a weak$^*$-neighborhood $U$ of $x^*$ such that $\overline{U}^*$ is disjoint from  $[\eta]'_{\varepsilon}(A) + \lambda B_{X^*}$. Then consider $V=\overline{U}^*+\lambda B_{X^*}$. It is clear that $V$ is a weak$^*$-closed neighborhood of $x^*$ such that $V$ is disjoint from $[\eta]'_{\varepsilon}(A)$. Then it follows from Lemma \ref{referee1} that $\eta(V\cap A)<\eps$. By the definition of $V$ we also have
$$ \overline{U}^* \cap (A+\lambda B_{X^*}) \subset (V \cap A)+ \lambda B_{X^*}.$$
This yields the estimate $\eta(\overline{U}^* \cap (A+\lambda B_{X^*})) < \varepsilon + \lambda b$. Therefore $x^*$ does not belong to
$[\eta]'_{\varepsilon+\lambda b}(A+\lambda B_{X^*})$, which finishes the proof of (\ref{eq1}).\\
Finally, this implies by iteration that $[\eta]^{\omega}_{\varepsilon+\lambda b}(A+\lambda B_{X^*})
=\emptyset$ whenever $[\eta]^{\omega}_{\varepsilon}(A)=\emptyset$, which yields property (iii) for $\eta^\omega$, with the same constant $b$ as for $\eta$.

\end{proof}

We will need the following elementary lemma.

\begin{Lem}\label{elementary} Let $X$ be a Banach space, $\eta$ a measure of non-compactness on $X^*$ and let $\eps'>\eps>0$. Then, for any weak$^*$-compact subset $A$ of $X^*$ we have
\[
[\eta^{\omega}]'_{\varepsilon'}(A) \subset [\eta]^{\omega}_{\varepsilon}(A)\subset [\eta^{\omega}]'_{\varepsilon}(A).
\]
Also, for any $n \in \Natural$, we have
\[
 [\eta^{\omega^n}]'_{\eps'}(A)\subset [\eta]^{\omega^n}_{\eps}(A)\subset [\eta^{\omega^n}]'_{\eps}(A).
\]

\end{Lem}

\begin{proof} If $x^* \in A \setminus  [\eta^{\omega}]'_{\varepsilon}(A)$ then there is a weak$^*$-neighborhood $U$ of $x^*$  such that $\eta^\omega(A\cap \overline{U}^*)<\eps$ and therefore  $[\eta]^\omega_{\varepsilon}(A \cap \overline{U}^*)=\emptyset$. This implies that $x^* \not \in [\eta]^{\omega}_{\varepsilon}(A)$.

On the other hand, if $x^* \in A \setminus [\eta]^{\omega}_{\varepsilon}(A)$ then there is $n \in {\Bbb N}$ such that
$x^* \not \in [\eta]^{n}_{\varepsilon}(A)$. If $U$ is a weak$^*$-neighborhood  of $x^*$ such that $\overline{U}^*
\cap [\eta]^{n}_{\varepsilon}(A)=\emptyset$, then $[\eta]^{n}_{\varepsilon}(A \cap \overline{U}^*)=\emptyset$ and so $\eta^\omega(A \cap \overline{U}^*)\le \eps<\eps'$ and $x^* \in A \setminus  [\eta^{\omega}]'_{\varepsilon'}(A)$.

Let us assume now that it has been proved that for some $n \in \Natural$,  $[\eta^{\omega^n}]'_{\eps'}(A) \subset [\eta]^{\omega^n}_\eps(A)$ for all $\eps'>\eps$ and all weak$^*$-compact $A$.
Since $\eta^{\omega^n}$ is a measure of non-compactness, we infer from the first statement of Lemma~\ref{elementary} and this inductive hypothesis that
$[\eta^{\omega^{n+1}}]'_{\eps'}(A) \subset [\eta^{\omega^n}]^\omega_{\eps''}(A) =
\bigcap_{m=1}^\infty [\eta^{\omega^n}]^m_{\eps''}(A) \subset
\bigcap_{m=1}^\infty [\eta]^{\omega^n\cdot m}_\eps(A) =
[\eta]^{\omega^{n+1}}_\eps(A)$
for all $\eps'>\eps''>\eps$ and any weak$^*$-compact subset $A$ of $X$.

Finally, when $[\eta]^{\omega^n}_\eps(A) \subset [\eta^{\omega^n}]'_\eps(A)$ for every weak$^*$-compact $A$ has been proved for $n \in \Natural$, we easily get that
$[\eta]^{\omega^{n+1}}_\eps(A) =
\bigcap_{m=1}^\infty [\eta]^{\omega^n\cdot m}_\eps(A) \subset
\bigcap_{m=1}^\infty [\eta^{\omega^n}]^m_\eps (A) =
[\eta^{\omega^n}]^\omega_\eps(A) \subset
[\eta^{\omega^{n+1}}]'_\eps(A)$
is true for every weak$^*$-compact $A$.
\end{proof}

We end this section with a lemma describing the link between the slice derivation and the fragment derivation.

\begin{Lem}\label{convexhull} Let $X$ be a Banach space, $\eta$ a measure of non-compactness on $X^*$ and $\eps>0$. Then, for any convex and weak$^*$-compact subset $K$ of $X^*$, we have that
\[
\langle \eta\rangle '_{\varepsilon}(K)=\wsconv[\eta]'_{\varepsilon}(K).
\]
\end{Lem}

\begin{proof} Since $\langle \eta\rangle '_{\varepsilon}(K)$ is convex, weak$^*$-closed and contains $[\eta]'_{\varepsilon}(K)$, it is clear that $\wsconv[\eta]'_{\varepsilon}(K) \subset \langle \eta\rangle '_{\varepsilon}(K)$.\\
Consider now $x^*\in K\setminus \wsconv[\eta]'_{\varepsilon}(K)$. It follows from the Hahn-Banach theorem that we can find a weak$^*$-open half space $H$ containing $x^*$ and so that $\overline{H}^*\cap \wsconv[\eta]'_{\varepsilon}(K)=\emptyset$. Let $S=K\cap \overline{H}^*$. By Lemma \ref{referee1}, we have that $\eta(S)<\eps$. Therefore $x^*$ does not belong to $\langle \eta\rangle '_{\varepsilon}(K)$, which concludes the proof of this lemma.
\end{proof}

\section{convexifiable and homogeneous measures of non-compactness}\label{convexifiable}

\begin{Def} Let $X$ be a Banach space. We say that a measure of non-compactness $\eta$ on $X^*$ is \emph{convexifiable} if there exists $\kappa \ge 1$ such that for any weak$^*$-compact subset $A$ of $X^*$, we have that $\eta(\wsconv(A)) \leq \kappa \eta(A)$. The infimum of the set of all constants $\kappa$ satisfying the above property (which also belongs to this set) is called the \emph{convexifiability constant} of $\eta$.
\end{Def}

\begin{Def} Let $X$ be a Banach space. A measure of non-compactness $\eta$ on $X^*$ is \emph{homogeneous} if for any weak$^*$-compact subset $A$ of $X^*$ and any $\lambda$ in $\Rdb$, we have $\eta(\lambda A)=\abs{\lambda}\eta(A)$.
\end{Def}

The following lemma is straightforward.

\begin{Lem}\label{homogeneous} Let $X$ be a Banach space and $B$ a weak$^*$-compact subset of $X^*$. Assume that $\eta$ is a homogeneous measure of non-compactness on $X^*$. Then for any $\eps>0$ and any $\lambda$ in $(0,+\infty)$, we have
$$[\eta]'_{\lambda \eps}(\lambda B)=\lambda [\eta]'_\eps(B)\ \ {\rm and}\ \ \langle \eta \rangle'_{\lambda \eps}(\lambda B)=\lambda \langle \eta \rangle'_\eps(B).$$
\end{Lem}

The following proposition is crucial.

\begin{Prop}\label{convex}
Let $X$ be a Banach space and $\eta$ a homogeneous and convexifiable measure of non-compactness on $X^*$ with convexifiability constant $\kappa$. Assume that $A$ is a weak$^*$-compact subset of $X^*$ such that $[\eta]'_\varepsilon(A) \subset \lambda A$ for some $\lambda \in (0,1)$ and $\varepsilon>0$. Then
\[
\forall\, \eps'>\eps\ \ \ \langle \eta\rangle^\omega_{\kappa\varepsilon'}(\wsconv(A)) =\emptyset.
\]
\end{Prop}

\begin{proof}
Let $\eps'>\eps$, fix $ \zeta \in (\lambda, 1)$ and take some $ \xi \in (\zeta, 1)$ whose precise value will be fixed later. Let us write $B:=\wsconv(A)$.
The key step of the proof will be to show that $\langle \eta \rangle'_{\kappa\eps'}(B)\subset \xi B$ for $\xi$ close enough to $1$. In order to do so we need to estimate the $\eta$-measure of weak$^*$-slices of $B$ which are disjoint from $\xi B$.
Once we observe that each such slice $S$ lies in a small neighborhood of the weak$^*$-closed convex hull $D$ of a well chosen $\eta$-small slice $K$ of $A$, we will be in a position to apply the property (iii) and the convexifiability of $\eta$.

Le us be more precise. Fix $x \in X$ such that $\sup_{x^* \in B}x^*(x)=1$ and consider the weak$^*$-closed half space  $H=\{x^*\in X^*,\ x^*(x) \ge\xi \}$. We denote $S=H \cap B$.\\
Let now $D=\wsconv(K)$ where $K=\{x^* \in A: x^*(x) \geq \zeta \}$. Since $K\cap [\eta]'_\varepsilon(A)=\emptyset$, it follows from the weak$^*$-compactness of $K$ and Lemma \ref{referee1} that $\eta(K)< \eps$. Since $\eta$ is convexifiable, we  have  $\eta(D) < \kappa \varepsilon$.\\
Note that $B=\overline{\rm conv}^*\big((A\cap \{x\le\zeta\})\cup (A\cap \{x\ge \zeta\})\big) \subset {\rm conv}(Q\cup D)$, where $Q=\{x\le\zeta\}\cap B$. In particular, any point  $x^* \in S$ can be written  $x^*=t x^*_1 +
(1-t) x^*_2$ with $x^*_1 \in D$, $x^*_2 \in Q$ and $t \in [0,1]$. If $x^*(x) > \xi$, an elementary computation shows that $t > t_\xi=(\xi -\zeta)(1-\zeta)^{-1}$.\\
Since $B$ is bounded, there is $\beta>0$ such that $B\subset \beta B_{X^*}$. Then we choose $\xi\in (\zeta, 1)$ such that
$$2b(1-t_\xi)\beta<\kappa(\eps'-\eps),$$
where $b>0$ is the constant given by the property (iii) of the measure of non-compactness $\eta$. Note that $\xi$ depends only on $\lambda,\ b,\ \beta,\ \kappa$ and $\eps'-\eps$.\\
We have shown that $S\subset [t_\xi,1]D+[0,1-t_\xi]Q$. Note that $[t_\xi,1]D\subset D+\beta(1-t_\xi)B_{X^*}$. Therefore, $S\subset D+2\beta(1-t_\xi)B_{X^*}$. We deduce that $\eta(S)\le \eta(D)+\kappa(\eps'-\eps)<\kappa \eps'$.\\
We have proved that under the assumptions of Proposition \ref{convex}, the following holds: there exists $\xi<1$ such that $\langle \eta \rangle '_{\kappa \eps'}(B)\subset \{x^*\in B,\ x^*(x)\le\xi\}$ whenever $\sup_B x=1$. Therefore, using the Hahn-Banach theorem for the weak$^*$ topology, we deduce that  $\langle \eta\rangle'_{\kappa\varepsilon'}(B)\subset \xi B$.\\
Finally, we combine an iteration of this argument with Lemma \ref{homogeneous} to get that for any $n\in \Ndb$, $\langle \eta\rangle^n_{\kappa\varepsilon'}(B)\subset \xi^n B$. Therefore, for $n\in \Ndb$ large enough, $\langle \eta\rangle^n_{\kappa\varepsilon'}(B)\subset \beta^{-1}b^{-1}\kappa\eps B\subset b^{-1}\kappa\eps B_{X^*}$. It follows that $\eta(\langle \eta\rangle^n_{\kappa\varepsilon'}(B))<\kappa\eps'$ and finally that $\langle \eta\rangle^{n+1}_{\kappa\varepsilon'}(B)=\emptyset$.
\end{proof}

\section{Sublinear measures of non-compactness.}\label{sublinear}

Property $(iii)$ in Definition \ref{def1} provides a control of the increase of the measure of a set after adding a ball to it. Actually the Kuratowski measure of non-compactness and its $\omega$-iterates have an even better behavior for general sums of sets. We shall introduce a definition.

\begin{Def} Let $X$ be a Banach space. A measure of non-compactness $\eta$ on $X^*$ is \emph{subadditive} if for all weak$^*$-compact subsets $A$ and $B$ of $X^*$, $\eta(A+B) \leq \eta(A)+\eta(B)$.\\
The measure of non-compactness $\eta$ on $X^*$ is \emph{sublinear} if it is homogeneous and subadditive.
\end{Def}

\begin{Remark} Note that property (i) of Definition \ref{def1} implies that a sublinear measure of non-compactness is translation invariant and that the sublinearity implies property (iii) in Definition \ref{def1} with $b=\eta(B_{X^*})$.
\end{Remark}

\begin{Exa}It is easily checked that the Kuratowski measure $\sigma$ on a dual Banach space $X^*$ is sublinear.
\end{Exa}
In order to show that the $\omega$-iterates of $\sigma$ are sublinear as well, we will prove first a few elementary facts.
Consider two dual Banach spaces $X^*_1$ and $X_2^*$ and measures of non-compactness $\eta_1$ and $\eta_2$ in each of them. The product $X^*_1 \times X^*_2$ is a dual Banach space endowed with the supremum norm. Consider the set function $\eta_1 \times \eta_2$ defined on $X^*_1 \times X^*_2$ by
$$ (\eta_1 \times \eta_2)(A) = \inf\{ \varepsilon>0,\ A \subset \bigcup_{i=1}^n A^1_i \times A^2_i\ \ {\rm with}\ \eta_j(A^j_i) < \varepsilon;\ {\rm for\ all}\ j\in\{1,2\}\ {\rm and}\ i\le n\}.$$
Since $\eta_1$ and $\eta_2$ are measures of non-compactness, it follows from property (ii) in Definition \ref{def1} that we actually have
$$(\eta_1 \times \eta_2)(A) = \inf\{ \varepsilon>0,\ A \subset  A^1 \times A^2\ \ {\rm with}\ \eta_j(A^j) < \varepsilon\ {\rm for}\ j=1,2\}.$$

\begin{Lem}\label{producto}
The function $\eta_1 \times \eta_2$ defined above is a measure of non-compactness defined on
$X^*_1 \times X^*_2$ and the following properties are satisfied:

\smallskip
(1) $[ \eta_1 \times \eta_2 ]'_{\varepsilon}(A^1 \times A^2) \subset
([ \eta_1 ]'_{\varepsilon}(A^1) \times A^2) \cup
(A^1 \times [ \eta_2 ]'_{\varepsilon}(A^2)).$

\smallskip
(2) If $[ \eta_1 ]^\omega_{\varepsilon}(A^1)=\emptyset$ and
$[ \eta_2 ]^\omega_{\varepsilon}(A^2)=\emptyset$, then
$[ \eta_1 \times \eta_2 ]^\omega_{\varepsilon}(A^1 \times A^2)=\emptyset$.
\end{Lem}

\begin{proof}
It is elementary to check the properties from Definition \ref{def1}, as well as the first statement. The second one follows by iteration of the previous set inclusion.
\end{proof}

\begin{Lem}\label{lipschitz}
Let $T: X^*_1 \rightarrow X^*_2$ be a weak$^*$-continuous linear operator. Suppose that there exists $\lambda>0$ such that $\eta_2(T(A)) \leq \lambda \eta_1(A)$ for any  weak$^*$- compact subset $A$ of $X^*_1$. Then, the following holds:

\smallskip
(1) If $A \subset X^*_1$, then $[ \eta_2 ]'_{\lambda \varepsilon}(T(A)) \subset T([ \eta_1 ]'_\varepsilon (A))$.

\smallskip
(2) If $[ \eta_1 ]^\omega_{\varepsilon}(A)=\emptyset$, then
$[ \eta_2 ]^\omega_{\lambda \varepsilon}(T(A))=\emptyset$.
\end{Lem}

\begin{proof}
We will only prove the first statement which clearly implies the second. So, let $y^* \in T(A) \setminus T([ \eta_1 ]'_\varepsilon (A))$ and fix $x^*\in A$ such that $T(x^*)=y^*$. Since $T$ is weak$^*$-continuous, $T([ \eta_1 ]'_\varepsilon (A))$ is weak$^*$-compact and there exists a weak$^*$-neighborhood $U$ of $y^*$ such that $\overline{U}^*$ is disjoint from $T([ \eta_1 ]'_\varepsilon (A))$. Thus $V=T^{-1}(\overline{U}^*)$ is weak$^*$-closed and is a weak$^*$-neighborhood of $x^*$ disjoint from $[ \eta_1 ]'_\varepsilon (A)$. Then Lemma \ref{referee1} insures that $\eta_1(A\cap V)<\eps$. It follows that $\eta_2(T(A)\cap \overline{U}^*)=\eta_2(T(A\cap V))<\lambda \eps$. Therefore, $y^*\notin [ \eta_2 ]'_{\lambda \varepsilon}(T(A))$. This concludes our proof.
\end{proof}

\begin{Prop}\label{additive}
Let $\eta$ be a sublinear measure of non-compactness defined on a dual Banach space $X^*$.
Then $\eta^\omega$ is also a sublinear measure of non-compactness.
\end{Prop}

\begin{proof} By applying Lemma \ref{lipschitz} to the operator $T$ defined by $T(x^*)=\lambda x^*$ for $x^*\in X^*$, we deduce immediately that $\eta^\omega$ is homogeneous.\\
Let now $A, B \subset X^*$ be weak$^*$-compact and $\varepsilon_1, \varepsilon_2 >0$ such that $\eta^\omega (A) < \varepsilon_1$ and $\eta^\omega (B) < \varepsilon_2$. Let $A' = \varepsilon^{-1}_1 A$ and $B' = \varepsilon^{-1}_2 B$. Since $\eta^\omega$ is homogeneous, we have that
$ [ \eta ]^\omega_{1}(A') =\emptyset$ and
$ [ \eta ]^\omega_{1}(B') =\emptyset$. It follows from Lemma \ref{producto} that $[ \eta \times \eta ]^\omega_{1}(A' \times B')=\emptyset$. Consider now the operator $T: X^* \times X^* \rightarrow X^*$ defined by $T(x^*,y^*)=\varepsilon_1 x^* + \varepsilon_2 y^*$. Since $\eta$ is sublinear, we may easily deduce that $\eta(T(C)) \leq (\varepsilon_1 + \varepsilon_2) (\eta \times \eta)(C)$ for any weak$^*$-compact subset $C$ of $X^* \times X^*$. In particular, we can apply Lemma \ref{lipschitz} to get that $[ \eta ]^\omega_{\varepsilon_1+\varepsilon_2}(A+B)= [ \eta ]^\omega_{\varepsilon_1+\varepsilon_2}(T(A' \times B')) =\emptyset$, that is, $\eta^{\omega}(A+B) \le \varepsilon_1 +\varepsilon_2$. Since $\varepsilon_1$ and $\varepsilon_2$ were arbitrary, we have proved the subadditivity of $\eta^\omega$.
\end{proof}

As an immediate consequence we have.

\begin{Cor}\label{KuraSub} Let $X$ be a Banach space and denote by $\sigma$ the Kuratowski measure of non-compactness on $X^*$. Then, for any $n\in \Ndb$, $\sigma^{\omega^n}$ is  a sublinear measure of non-compactness on $X^*$.
\end{Cor}

Let us recall the notation that we use for the slice derivation
$$ \langle \eta\rangle '_{\varepsilon}(A)=\{ x^* \in A,\ {\rm for\ any}\ {\rm weak^*-open\ halfspace}\ H\ {\rm containing}\ x^*,\   \eta(A \cap \overline{H}^*) \ge\varepsilon\}.$$
Along the remaining part of this section we shall only deal with weak$^*$-compact convex sets, since the effect of the slice derivation on convex sets produces convex sets again. Also, \emph{ we will assume that the measure of non-compactness $\eta$ is sublinear}. Thus we will be able to keep better constants in the formulas.

\begin{Lem}\label{fat} Let $\eta$ be a sublinear measure of non-compactness defined on a dual Banach space $X^*$. If $A$ and $C$ are  weak$^*$-compact and convex then
$$ \langle \eta \rangle'_{\varepsilon + \eta(C)} (A + C) \subset
\langle \eta \rangle'_{\varepsilon} (A) + C $$
\end{Lem}

\begin{proof} Let $x^*\in (A+C)\setminus \langle \eta \rangle'_{\varepsilon} (A) + C $. Then there exist $x\in X$ and $a\in \Rdb$ such that $x^*(x)<a$ and $S\cap \big(\langle \eta \rangle'_{\varepsilon} (A) + C\big)=\emptyset$, where $S=\{y^*\in X^*,\ y^*(x)\le a\}$. Consider now
$$T=\{y^*\in X^*,\ y^*(x)\le a-\inf_{u^*\in C}u^*(x)\}.$$
Then $T\cap \langle \eta \rangle'_{\varepsilon} (A)=\emptyset$ and Lemma \ref{referee1} yields that $\eta(T\cap A)<\eps$.\\
Finally, since $S \cap (A+C)\subset (T\cap A)+C$, we get that $\eta(S \cap (A+C))<\eps+\eta(C)$ and that $x^*\notin \langle \eta \rangle'_{\varepsilon + \eta(C)} (A + C)$, which concludes our proof.

\end{proof}

We will need a modification of the slice derivation. Given a convex
weak$^*$-compact set $D$ consider
$$ \langle \eta\rangle '_{\varepsilon}(A|_D)=\{ x^* \in A, \forall H  \ {\rm weak}^*{\rm -open\ halfspace},\ x^* \in H,\   \overline{H}^* \cap D =\emptyset \Rightarrow  \eta(A \cap \overline{H}^*) \ge\varepsilon\}.$$

\begin{Lem}\label{convexy} Let $\eta$ be a sublinear measure of non-compactness defined on a dual Banach space $X^*$. Suppose that $C$ and $D$ are convex weak$^*$-compact subsets of $X^*$ so that $\eta(D) \leq 1$ and $\eta(C) \leq \varepsilon$ with  $\varepsilon \in (0,1)$. For $\delta \in (0,1)$, the sequence $(A_k)_{k=1}^\infty$ defined recursively by $A_1= \overline{conv}^*(C \cup D)$ and
$A_{k+1}= \overline{\mbox{conv}}^*((C \cap \langle \eta \rangle '_{\varepsilon+\delta}(A_{k}|_D)) \cup D)$ satisfies
$$\forall k\ge 2\ \  \sup\{f,A_{k}\} - \sup\{f,D\} \leq  (1-\delta/2)^{k-1}(\sup\{f,A_1\}-\sup\{f,D\}) $$
for every functional $f \in X$ such that $\sup\{f,D\} \leq \sup\{f,C\}$.
\end{Lem}

\begin{proof}
With small modifications, it is essentially done in \cite{Raja2010}.
It is enough to prove the inequality for the first step
$$ \sup\{f,A_2\} -  \sup\{f,D\}  \leq  \big(1-\frac{\delta}{2}\big)(\sup\{f,A_1\}- \sup\{f,D\} ).$$
Consider
$$ E=\{ (1-\lambda)y^* +\lambda z^*: y^* \in C, z^* \in D, \lambda \in \big[\frac{\delta}{2},1\big] \} $$
Note that $E$ contains $D$ and is weak$^*$-closed and convex.
If $x^* \in A_1 \setminus E$ then $x^*=(1-\lambda)y^* +\lambda z^*$ with $y^* \in C$, $z^* \in D$ and
$\lambda \in [0,\frac{\delta}{2}]$.
Since $x^*-y^*=\lambda(z^*-y^*)$, we have
$$A_1 \setminus E\ \subset\ C + \bigcup_{\lambda \in [0,\frac{\delta}{2}]} \lambda(D-C).$$
Using the compactness of $[0,\frac{\delta}{2}]$, it follows that for every $\nu>0$, there exists a finite subset $F$ of $[0,\frac{\delta}{2}]$ such that
$$A_1 \setminus E\ \subset\ C + \bigcup_{\lambda \in F} \big(\lambda(D-C)+\nu B_{X^*}\big).$$
The set on the right hand side of the above inclusion is weak$^*$-closed, so we deduce from the properties of $\eta$ that for all $\nu>0$,
$$\eta(\overline{A_1 \setminus E}^*) \leq \eta(C) + \delta+ \nu \eta(B_{X^*})$$
and therefore that $\eta(\overline{A_1 \setminus E}^*) \leq \eps+\delta$.\\
This implies that any weak$^*$-closed slice of $A_1$ disjoint from $E$ has $\eta$-measure less than
$\varepsilon+\delta$. Therefore we have $\langle \eta \rangle'_{\varepsilon+\delta}(A_1|_D) \subset E$ and thus
$\sup\{f,A_2\}  \leq \sup\{f,E\}$.
Moreover, we have
$$\sup\{f,E\} -  \sup\{f,D\}  \leq \big(1-\frac{\delta}{2}\big)\sup\{f,C\} + \frac{\delta}{2} \sup\{f,D\} -  \sup\{f,D\}  $$
$$ =  \big(1-\frac{\delta}{2}\big)\sup\{f,C\} + \big(\frac{\delta}{2}-1\big) \sup\{f,D\}  = \big(1-\frac{\delta}{2}\big)(\sup\{f,C\} -  \sup\{f,D\} ),$$
which concludes our proof.
\end{proof}

\begin{Lem} \label{recursion} Let $\eta$ be a sublinear measure of non-compactness defined on a dual Banach space $X^*$.
For every $\varepsilon, \delta>0$, every convex weak$^*$-compact subset $A$ of $X^*$ and every weak$^*$-open halfspace $H$ we have
\[
 \langle \eta \rangle_\varepsilon^\omega(\overline{H}^* \cap A)= \emptyset \Rightarrow \langle \eta \rangle_{\varepsilon+\delta}^\omega(A) \subset {A\setminus H}.
\]

\end{Lem}

\begin{proof} Since the measure $\eta$ is homogeneous, we may assume without loss of generality that $\eta(A)\le 1$. Then we can also assume that $\eps,\delta$ are in $(0,1)$, since otherwise $\langle \eta \rangle_{\varepsilon+\delta}^\omega(A)=\emptyset$. In fact, we are going to prove a more precise statement. Namely, for every $\varepsilon, \delta, \zeta \in (0,1)$ and $n \in {\Bbb N}$, there exists  $N=N(\varepsilon,\delta, \zeta, n)$ such that whenever $A$ is convex weak$^*$-compact with $\eta(A) \leq 1$ and $H$ is a weak$^*$-open halfspace we have
\begin{equation}\label{rec}
\langle \eta \rangle_\varepsilon^n (\overline{H}^* \cap A)=\emptyset \Rightarrow
 \langle \eta \rangle_{\varepsilon+\delta}^N(A) \subset (A \setminus H) + (\zeta/2) (A-A).
\end{equation}
Let $B=\frac{1}{2}(A-A)$, so $\eta(B) \leq 1$. We shall use an inductive argument on $n\in \Ndb$ to prove the result. For $n=1$ the result is true, even with $\delta=\zeta=0$. Indeed,  if $\langle \eta \rangle_\varepsilon' (\overline{H}^* \cap A)=\emptyset$, the usual compactness argument implies that $\eta(\overline{H}^* \cap A)<\eps$ and therefore that $\langle \eta \rangle_{\varepsilon}'(A) \subset A \setminus H$.\\
Suppose now that it is true for some $n\ge1$ and let $H$ be a weak$^*$-open halfspace with
$\langle \eta \rangle_\varepsilon^{n+1}(\overline{H}^* \cap A)=\emptyset$ and
$\langle \eta \rangle_\varepsilon^n (\overline{H}^* \cap A) \not = \emptyset$.
Fix $p \in {\Bbb N}$ such that $(1-\delta/4)^{p-1} \leq \zeta/4$ (this choice only depends on $\delta$ and $\zeta$). We will use Lemma \ref{convexy} with $C=\langle \eta \rangle_\varepsilon^n (\overline{H}^* \cap A)$, $D=A \setminus H$  and $\delta/2$ instead of $\delta$. Then, if $(A_k)_{k=1}^\infty$ is defined as in Lemma \ref{convexy}, we obtain that
\begin{equation}\label{claim}
A_p \subset (A \setminus H)+ \frac{\zeta}{2} B .
\end{equation}
Indeed, by the Hahn-Banach Theorem, it is enough to show that
$$ \sup\{ g,A_p\} \leq \sup\{g, A \setminus H\} + \frac{\zeta}{2} $$
for every $g \in X \setminus \{0\}$ such  that $\sup\{g,B\}=1$.
Suppose that it is not the case. Then we have $ \sup\{ g,A_p\} > \sup\{g, A \setminus H\} $ and so $\sup\{g,C\} > \sup\{g,D\}$. On the other hand, $\eta(D)\le \eta(A)\le 1$ and $\langle \eta\rangle'_\eps (C)=\emptyset$ implies that $\eta(C)<\eps$. So, by Lemma \ref{convexy}, for our choice of $p$, we have
$$\sup\{g,A_p\} - \sup\{g,D\} \le \frac{\zeta}{4}\big(\sup\{g,A\} - \sup\{g,A \setminus H\}\big) \leq \frac{\zeta}{2}, $$
which leads to a contradiction.\\
Set first $A_0=A$. Assume $0\le k\le p-1$ and consider now $G$ a weak$^*$-open halfspace such that $\overline{G}^*\cap A_{k+1}$ is empty. Since $D \subset A_{k+1}$ we have $\overline{G}^* \cap A_k \cap D=\emptyset$, so
$\overline{G}^* \cap A_k \subset \overline{H}^* \cap A$ and thus $\langle \eta \rangle_\varepsilon^n (\overline{G}^* \cap A_k) \subset C$. In particular, $\langle \eta \rangle_\varepsilon^n (\overline{G}^* \cap \langle \eta \rangle'_{\varepsilon+\delta/2}(A_k|_D))$ is a subset of $C$. So
$$\langle \eta \rangle_\varepsilon^n \big(\overline{G}^* \cap \langle \eta \rangle'_{\varepsilon+\delta/2}(A_k|_D)\big)\ \subset \  \overline{G}^* \cap C \cap  \langle \eta \rangle'_{\varepsilon+\delta/2}(A_k|_D)\
\subset \ \overline{G}^* \cap A_{k+1} = \emptyset. $$
We can now deduce from our induction hypothesis that for any $\xi\in (0,1)$
$$ \langle \eta \rangle_{\varepsilon+\delta/2}^{1+N(\varepsilon,\delta/2,\xi,n)}(A_k) \subset (\langle \eta \rangle'_{\varepsilon+\delta/2}(A_k|_D)\setminus G)+\xi B\ \subset\ (A_k\setminus G)+\xi B.$$
The above inclusion being true for any weak$^*$-open halfspace such that $\overline{G}^*\cap A_{k+1}=\emptyset$, it follows from the Hahn-Banach Theorem that
$$\forall \xi \in (0,1)\ \ \langle \eta \rangle_{\varepsilon+\delta/2}^{1+N(\varepsilon,\delta/2,\xi,n)}(A_k)
\subset A_{k+1} + \xi B.$$
Pick now $\xi\in (0,1)$ so that $2p\xi\le \zeta$ and $2p\xi\le \delta$ and let $m = 1+N(\varepsilon,\delta/2,\xi,n)$. Mixing the former inclusion with Lemma \ref{fat} gives
$$ \langle \eta \rangle^m_{\varepsilon +\delta}(A_k + k \xi B)\ \subset\ \langle \eta \rangle^m_{\varepsilon +\frac{\delta}{2}+k\xi}(A_k + k \xi B)\ \subset \langle \eta \rangle^m_{\varepsilon+\frac{\delta}{2}}(A_k) + k\xi B\ \subset\ A_{k+1} + (k+1) \xi B.$$
Chaining these inclusions from $k=0$ to $k=p-1$ and using (\ref{claim}) we get that
$$ \langle \eta \rangle^{pm}_{\varepsilon+\delta}(A)\ \subset\ A_p + (\zeta/2) B\ \subset\  A \setminus H + \zeta B.$$
This concludes the proof of (\ref{rec}).
\end{proof}

The next proposition, which is a version of Lemma~\ref{elementary} for slice derivations, is a consequence of the previous results. We shall see later that it applies to the $\omega$-iterates of the Kuratowski measure.

\begin{Prop}\label{laststep} Let $\eta$ be a sublinear measure of non-compactness defined on a dual Banach space $X^*$. Assume that the measure $\eta$ satisfies the following additional property: there exists a constant
$\theta>0$ such that for any convex weak$^*$-compact subset $A$ of $X^*$, $[\eta]_\varepsilon^\omega(A)=\emptyset$ implies that
$\langle \eta \rangle_{\theta \varepsilon}^\omega(A)=\emptyset$.\\
Then for any convex weak$^*$-compact subset $A$ of $X^*$, any  $\varepsilon>0$, any $\lambda>\theta$ and every ordinal $\alpha$ we have
$$ \langle \eta \rangle^{\omega.\alpha}_{\lambda \varepsilon}(A) \subset \langle \eta^{\omega} \rangle^\alpha_{\varepsilon} (A). $$
\end{Prop}

\begin{proof}
It is enough to show that
$$ \langle \eta \rangle^{\omega}_{\lambda \varepsilon}(A) \subset
\langle \eta^{\omega} \rangle'_{\varepsilon} (A) $$
since the general statement follows easily by iteration.
Fix $0 < \delta < \varepsilon (\lambda-\theta)$. For any weak$^*$-open halfspace $H$ such that $\overline{H}^* \cap \langle \eta^{\omega} \rangle'_{\varepsilon} (A) =\emptyset$ we have
$ \langle \eta^{\omega} \rangle'_{\varepsilon} (\overline{H}^* \cap A) =\emptyset$ and so
$[ \eta ]^\omega_{\varepsilon} (\overline{H}^* \cap A) =\emptyset$.
By the assumption, $\langle \eta \rangle^\omega_{\theta \varepsilon} (\overline{H}^* \cap A) =\emptyset$ and by
Lemma~\ref{recursion} we have
$$\langle \eta \rangle^\omega_{\lambda \varepsilon} (A) \subset
\langle \eta \rangle^\omega_{\theta \varepsilon+\delta} (A) \subset A \setminus H.$$
Since $H$ was arbitrary, we get that
$ \langle \eta \rangle^\omega_{\lambda \varepsilon} (A) \subset \langle \eta^{\omega} \rangle'_{\varepsilon} (A)  $ as we wanted.
\end{proof}

\section{Application to the Kuratowski measure of non-compactness.}\label{kuratowski}

In this section we will show that $\sigma^{\omega^n}$ is convexifiable for every $n\geq 0$ and then use it together with the sublinearity of $\sigma^{\omega^n}$ in order to compare the Szlenk index and the convex Szlenk index of a Banach space.

\medskip

First, we wish to recall some definitions.

\begin{Def}
Let $X$ be a Banach space, $K$ a weak$^*$-compact subset of $X^*$ and $\eps>0$.\\
We define $Sz(K,\eps)=\inf\{\alpha,\ [\sigma]^\alpha_\eps(K)=\emptyset\}$ if it exists and $Sz(K,\eps)=\infty$ otherwise.
Then $Sz(K):=\sup_{\eps>0}Sz(K,\eps)$.\\
Further we define $Kz(K,\eps)=\inf\{\alpha,\ \langle \sigma\rangle ^\alpha_\eps(K)=\emptyset\}$ if it exists and $Kz(K,\eps)=\infty$ otherwise.
We put $Kz(K):=\sup_{\eps>0}Kz(K,\eps)$\\
Denote now
$$A^0_\eps:=K \ \  A^{\beta+1}_\eps:=\overline{\rm conv}^*\big([\sigma]'_\eps(A^\beta_\eps)\big)\ \ {\rm and}\ \
A^{\beta}_\eps:=\bigcap_{\gamma<\beta}A^{\gamma}_\eps\ {\rm if\ \beta\ is\ a\ limit\ ordinal}.$$
We set $Cz(K,\eps)=\inf\{\beta,\ A^{\beta}_\eps=\emptyset\}$ if it exists and $Cz(K,\eps)=\infty$ otherwise,
and $Cz(K):=\sup_{\eps>0}Cz(K,\eps)$.\\
Finally we denote $Sz(X)=Sz(B_{X^*})$ and $Cz(X)=Cz(B_{X^*})$. The ordinal $Sz(X)$  is called the \emph{Szlenk index} of $X$ and $Cz(X)$ is called the \emph{convex Szlenk index} of $X$.
\end{Def}

Note that for any ordinal $\alpha$ and all $\eps>0$
\begin{equation}\label{s-sigma}
[\sigma]_{2\varepsilon}^\alpha(K)\subset s_\varepsilon^\alpha(K) \subset [\sigma]_{\varepsilon/2}^\alpha(K).
\end{equation}
It follows that the above definitions of the Szlenk index and the convex Szlenk index of a Banach space actually coincide with those given in the introduction. Moreover, it follows from Lemma \ref{convexhull} that for any weak$^*$-compact and convex subset $K$ of $X^*$, $Cz(K)=Kz(K)$ and it is clear that $Cz(K)=Kz(K)\ge Sz(K)$.

\medskip

Let us first state an elementary fact, well known for the case $\eta=\sigma$ and $K=B_{X^*}$ (see \cite{Sersouri1989} for the original idea, or \cite{Lancien2006}).

\begin{Lem}\label{formSz} Let $K$ be a convex weak$^*$-compact subset of a dual Banach space $X^*$ and let $\eta$ be a homogeneous and translation invariant measure of non-compactness on $X^*$. Then

\smallskip
(i) For any $\eps>0$ and any ordinal $\alpha$, \  $\frac12 [\eta]^\alpha_\eps(K)+ \frac12 K \subset [\eta]^\alpha_{\eps/2}(K).$

\smallskip
(ii) For any $\eps>0$, $n\in \Ndb$ and any ordinal $\alpha$, \ $[\eta]_\eps^{\alpha}(K) \subset [\eta]_{\eps/2^n}^{\alpha.2^n}(K)$.

\smallskip
(iii) If $Sz(K)\neq \infty$ and $K\neq \emptyset$, then there exists an ordinal $\alpha$ such that $Sz(K)=\omega^\alpha$.
\end{Lem}

\begin{proof} (i) Since $K$ is convex, the statement is clearly true for $\alpha=0$. It also passes easily to limit ordinals. So assume that it is satisfied for some ordinal $\alpha$ and consider $x^*\notin [\eta]^{\alpha+1}_{\eps/2}(K)$. Assume, as we may, that $x^*$ belongs to $\frac12 [\eta]^\alpha_\eps(K)+ \frac12 K$ and thus by induction hypothesis to $[\eta]^\alpha_{\eps/2}(K)$. Then there exists a weak$^*$-neighborhood $U$ of $x^*$ such that $\eta([\eta]_{\varepsilon/2}^\alpha(K) \cap \overline{U}^*)<\eps/2$. For any  $u^*\in [\eta]^\alpha_\eps(K)$ and $v^*\in K$ so that $2x^*=u^*+v^*$,  we have that $W=2U-v^*$ is a weak$^*$-neighborhood  of $u^*$ such that $\eta(\overline{W}^*\cap [\eta]^\alpha_\eps(K))<\eps$. Indeed, we have by the definition of $W$ and the inductive hypothesis that $\frac12 (\overline{W}^*\cap [\eta]^\alpha_\eps(K))+\frac12 v^* \subset \overline{U}^* \cap [\eta]^\alpha_{\eps/2}(K)$, and $\eta$ is homogeneous and translation invariant. This shows that $x^*\notin \frac12 [\eta]^{\alpha+1}_\eps(K)+ \frac12 K$.

\smallskip (ii) It is enough to show that $[\eta]_\eps^{\alpha}(K) \subset [\eta]_{\eps/2}^{\alpha.2}(K)$. So let $x^*\in [\eta]_\eps^{\alpha}(K)$. It follows from (i) that $\frac12 x^*+\frac12 K \subset [\eta]^\alpha_{\eps/2}(K)$. Since $\eta$ is translation invariant we deduce that $\frac12 x^*+[\eta]^\alpha_{\eps/2}(\frac12 K) \subset [\eta]^{\alpha.2}_{\eps/2}(K)$. Finally we use the homogeneity of $\eta$ to get that $\frac12 x^*\in [\eta]^\alpha_{\eps/2}(\frac12 K)$ and therefore that $x^*\in [\eta]^{\alpha.2}_{\eps/2}(K)$.

\smallskip (iii) This is an easy consequence of property (ii) applied to the Kuratowski measure of non-compactness.

\end{proof}

We will need to use special families of trees on $\Ndb$. Let us first recall the basic notation and definitions about trees on $\Ndb$.\\
We denote $\Ndb^{<\omega}=\bigcup_{k=1}^\infty \Ndb^k \ \cup \{\emptyset\}$, where $\emptyset$ denotes the empty sequence. For $a\in \Ndb^{<\omega}$, we denote by $|a|$ the length of $a$, which is defined by $|a|=0$ if $a=\emptyset$ and $|a|=k$ if $a=(n_1,..,n_k)\in \Ndb^k$. There is a natural order $\le$ on $\Ndb^{<\omega}$ defined as follows: for $a,b\in \Ndb^{<\omega}$, $a\le b$ if $a=\emptyset$ or $b=(n_1,..,n_j)\in \Ndb^j$ with $j\in \Ndb$ and $a=(n_1,..,n_k)$ with $k\le j$. For $a,b \in \Ndb^{<\omega}$, we say that $b$ is a \emph{successor} of $a$, or that $a$ is the \emph{predecessor} of $b$ if $a\le b$ and $|b|=|a|+1$. If $a=(n_1,..,n_k)$ and $b=(m_1,..,m_j)$, we denote by $a^\smallfrown b$ the sequence $(n_1,..,n_k,m_1,..,m_j)$ (and $a^\smallfrown b=b$ if $a=\emptyset$, $a^\smallfrown b=a$ if $b=\emptyset$). For $a\in \Ndb^{<\omega}$ and $S$ a subset of $\Ndb^{<\omega}$, $a^\smallfrown S$ denotes the set $\{a^\smallfrown b,\ b\in S\}$.\\
A subset $T$ of $\Ndb^{<\omega}$ is a \emph{tree} on $\Ndb$ if for any $a$ in $T$ and any $b$ in $\Ndb^{<\omega}$ such that $b\le a$, we have that $b\in T$. A subset $B$ of a tree $T$ is a \emph{branch} of $T$ if it is a maximal totally ordered subset of $T$. For any tree $T$ on $\Ndb$, its derivative is $T'=T^1=\set{s \in T: s^\smallfrown (n) \in T \mbox{ for some }n \in \Natural}$. Then $T^\alpha$ is defined inductively for $\alpha$ ordinal as follows: $T^{\alpha+1}=(T^\alpha)'$ and $T^\alpha=\bigcap_{\beta<\alpha}T^\beta$ if $\alpha$ is a limit ordinal. A tree $T$ is said to be \emph{well founded} if there exists an ordinal $\alpha$ such that $T^\alpha=\emptyset$, or equivalently if all its branches are finite. If $T$ is a well founded tree on $\Ndb$, then its \emph{height} is the infimum of all $\alpha$ so that $T^\alpha=\emptyset$ and is denoted $o(T)$. Note that the height of a non empty well founded tree on $\Ndb$ is always a countable successor ordinal and that $T^\alpha=\{\emptyset\}$ if $o(T)=\alpha+1$.

\medskip We are now ready to define our families of trees on $\Ndb$.

\begin{Def}
For each ordinal $\alpha<\omega_1$ we define a family of trees $\cal T_\alpha$ as follows. We set ${\cal T}_0:=\{\{\emptyset\}\}$. Let now $\alpha$ be a countable ordinal such that $\alpha\ge1$.\\
If $\alpha=\beta+1$ is a successor ordinal we say that $T \in \cal T_\alpha$ if there exists an increasing sequence $(n_k)_{k=0}^\infty$ in $\Ndb$ and a sequence $(T_k)_{k=0}^\infty$ in ${\cal T}_{\beta}$ such that
$$T=\{\emptyset\}\cup  \bigcup_{k=0}^\infty (n_k)^\smallfrown T_{k}.$$
If $\alpha$ is a limit ordinal we say that $T \in \cal T_\alpha$ if there exists an increasing sequence $(n_k)_{k=0}^\infty$ in $\Ndb$, an increasing sequence $(\alpha_k)_{k=0}^\infty$ in $[0,\alpha)$ such that $\alpha_k \nearrow \alpha$ and a sequence $(T_k)_{k=0}^\infty$ such that $T_k \in {\cal T}_{\alpha_k}$ for all $k$ and
$$T=\{\emptyset\}\cup  \bigcup_{k=0}^\infty (n_k)^\smallfrown T_{k}.$$
\end{Def}

\begin{Remark} One can easily verify that for all  $\alpha<\omega_1$, $\cal T_\alpha$ is indeed a family of trees on $\Ndb$ and that for all $T \in {\cal T}_\alpha$, $o(T)=\alpha+1$.
\end{Remark}

\begin{Def} Let $T \in {\cal T}_\alpha$ for some ordinal $\alpha<\omega_1$. Note that for $s\in T^1$, there exists an increasing sequence in $\Ndb$ that we denote $(n_k^s)_{k=1}^\infty$ such that the set of successors of $s$ in $T$ is $\{s^\smallfrown (n_k^s),\ k\in \Ndb\}$. Note also that, if $s\in T^{\beta+1}$ for some $\beta \in [1,\alpha)$, then there exists $k_0\in \Ndb$ such that $s^\smallfrown (n_k^s) \in T^\beta$, for all $k\ge k_0$.\\
Then, we say that a family $(x^*_s)_{s \in T} \subset X^*$ is \emph{weak$^*$-continuous} if $x^*_{s\smallfrown (n_k^s)} \wstoo x^*_s$ as $k \to \infty$ for all $s \in T^1$. We say that it is \emph{$\eps$-separated} if $\|x^*_s-x^*_{s \smallfrown (n_k^s)}\|\geq \eps$ for all $s \in T^1$ and all $k \in \Natural$.
\end{Def}

Our first lemma follows from (\ref{s-sigma}) and the classical characterization of the Szlenk index in the separable case (see Lemma 3.4 in \cite{Lancien1996} for a non-separable version).

\begin{Lem}\label{treeCharSzlenk} Let $X$ be a separable Banach space, $K$ a weak$^*$-compact subset of $X^*$,  $\eps>0$ and $\alpha<\omega_1$.

\smallskip
(i) If $x^* \in [\sigma]_{\varepsilon}^\alpha(K)$, then for any $\rho<\frac{\eps}{4}$ there is $T\in {\cal T}_\alpha$ and a family $(x^*_s)_{s \in T} \subset K$ which is weak$^*$-continuous and  $\rho$-separated, such that $x^*_{\emptyset}=x^*$.

\smallskip
(ii) If there exists $T\in {\cal T}_\alpha$ and a family $(x^*_s)_{s \in T} \subset K$ which is weak$^*$-continuous and  $\eps$-separated, then $x^*_\emptyset \in [\sigma]_{\varepsilon}^\alpha(K)$.
\end{Lem}

\begin{proof} (i) Let $x^* \in [\sigma]_{\varepsilon}^\alpha(K)$. By (\ref{s-sigma}), we have that $x^*\in s^\alpha_{\eps/2}(K)$. Then it is easy to show by transfinite induction on $\alpha<\omega_1$ that if $x^*\in s^\alpha_\delta(K)$ for some $\delta >0$, then for any $\rho<\frac{\delta}{2}$ there is $T\in {\cal T}_\alpha$ and a family $(x^*_s)_{s \in T} \subset K$ which is weak$^*$-continuous and  $\rho$-separated, such that $x^*_{\emptyset}=x^*$.

(ii) One can prove by induction on $\beta\le \alpha$ that if $s\in T^\beta$, then $x^*_s\in [\sigma]_\eps^\beta(K)$. Alternatively, (ii) can also be proved directly by a transfinite induction on $\alpha$.
\end{proof}

We need the following property of weak$^*$-continuous separated trees in $X^*$.

\begin{Lem}\label{extract} Let $X$ be a separable Banach space, $\alpha<\omega_1$ and $T\in {\cal T}_\alpha$. Assume that $(x^*_s)_{s \in T}$ is a weak$^*$-continuous and $\eps$-separated family in $X^*$ and that $K$ is a weak$^*$-compact subset of  $X^*$ such that, for some $0<a\le b<\infty$ we have:
$$\forall s\in T\ \ \exists \lambda_s\in [a,b]\ \ \lambda_sx^*_s\in K.$$
Then there exists $\lambda \in [a,b]$ such that $\lambda x^*_{\emptyset} \in [\sigma]^\alpha_{a \varepsilon}(K)$ and for any $\nu>0$ there exists $S\subset T$ so that $S\in {\cal T}_\alpha$ and $\abs{\lambda_s-\lambda}<\nu$ for all $s\in S$.
\end{Lem}

\begin{proof} The proof is a transfinite induction on $\alpha<\omega_1$. The statement is clearly true for $\alpha=0$. So let us assume that it is satisfied for all $\beta<\alpha$.\\
Assume first that $\alpha$ is a limit ordinal. Then $T=\{\emptyset\}\cup \bigcup_{k=0}^\infty (n_k)^\smallfrown T_{k}$
where $T_{k} \in {\cal T}_{\alpha_k}$ for each $k \in \Ndb$, with $n_k \nearrow \infty$ and
$\alpha_k\nearrow \alpha$.
By our induction hypothesis, for all $k\in \Ndb$, there exists $\lambda_k\in [a,b]$ such that $\lambda_kx^*_{(n_k)}\in [\sigma]^{\alpha_k}_{a \varepsilon}(K)$ and for any $\nu>0$ there exists $S_k\subset T_k$ so that $S_{k} \in {\cal T}_{\alpha_k}$ and for all $s\in S_k$, $\abs{\lambda_s-\lambda_k}<\frac{\nu}{2}$.
By taking a subsequence, we may assume that $\lambda_k \to \lambda \in [a,b]$ and for all $k$, $\abs{\lambda-\lambda_k}<\frac{\nu}{2}$.
Then $\lambda_kx^*_{(n_k)}\wstoo \lambda x^*_\emptyset$.
Since the sets $[\sigma]^{\alpha_k}_{a \varepsilon}(K)$ are weak$^*$-closed, we get that $\lambda x^*_\emptyset$ belongs to their intersection and therefore to $[\sigma]^\alpha_{a \varepsilon}(K)$.
Moreover, $S=\{\emptyset\} \cup \bigcup_{k=0}^\infty (n_k)^\smallfrown S_{k}$ is a subset of $T$ belonging to ${\cal T}_\alpha$ such that for all $s\in S$, $\abs{\lambda-\lambda_s}<\nu$.

Assume now that $\alpha=\beta+1$. Then $T=\{\emptyset\} \cup \bigcup_{k=0}^\infty (n_k)^\smallfrown T_{k}$
where $T_{k} \in {\cal T}_{\beta}$ for each $k \in \Ndb$. By our induction hypothesis, for all $k\in \Ndb$, there exists $\lambda_k\in [a,b]$ such that $\lambda_kx^*_{(n_k)}\in [\sigma]^{\beta}_{a \varepsilon}(K)$ and for any $\nu>0$ there exists $S_k\subset T_k$ so that $S_{k} \in {\cal T}_{\beta}$ and for all $s\in S_k$, $\abs{\lambda-\lambda_s}<\frac{\nu}{2}$.
By taking a subsequence, we may assume that $\lambda_k \to \lambda \in [a,b]$ and for all $k$, $\abs{\lambda-\lambda_k}<\frac{\nu}{2}$.
Then $\lambda_kx^*_{(n_k)}\wstoo \lambda x^*_\emptyset$ and $\liminf_k\|\lambda_kx^*_{(n_k)}-\lambda x^*_\emptyset\|\ge a\eps$. Therefore $\lambda x^*_\emptyset \in [\sigma]^{\beta+1}_{a \varepsilon}(K)=[\sigma]^\alpha_{a \varepsilon}(K)$.
We also have that $S=\{\emptyset\}\cup  \bigcup_{k=0}^\infty (n_k)^\smallfrown S_{k}$ is a subset of $T$ belonging to ${\cal T}_\alpha$ such that for all $s\in S$, $\abs{\lambda-\lambda_s}<\nu$. This finishes our induction.
\end{proof}

We now deduce the following.

\begin{Prop}\label{radial} Let $X$ be a separable Banach space and $A$ be a weak$^*$-compact subset of $X^*$ such that $[\sigma^{\omega^n}]^{m}_{\varepsilon}(A)=\emptyset$ for some integers $n\ge 0$ and $m\ge 1$. Then there is a symmetric radial weak$^*$-compact set $B$ containing $A$ and such that
\[
[\sigma^{\omega^n}]'_{7\varepsilon}(B) \subset \Big(1-\frac{1}{32(m+1)}\Big)B.
\]
\end{Prop}

\begin{proof}
Considering $-A\cup A$ instead of $A$ and we may assume, without loss of generality, that $A$ is symmetric.
Fix $r \in (0,1)$ such that $3(r+r^2)>4$ (for instance $r=7/8$).
Now define the sets
$$ B_k=
\{ \lambda x^* : x^* \in [\sigma^{\omega^n}]^k_\varepsilon(A) \cup r A,\
\lambda \in [0,1] \} $$
for $k=0,\dots,m$. The sets $B_k$ are clearly symmetric and radial. Using the weak$^*$-compactness of $rA$ and $[\sigma^{\omega^n}]^k_\varepsilon(A)$ it is not difficult to see that they are also weak$^*$-compact. Therefore, we can define the Minkowski functional $f_k$ of $B_k$ which is  weak$^*$ lower semi-continuous. Notice that $f_k \leq f_{k+1}$ and $f_{m} =r^{-1} f_0$. We now define
$$f(x^*)=\frac{1}{2} f_0(x^*) + \frac{r}{2(m+1)} \sum_{k=0}^{m} f_k(x^*)$$
Clearly $f \leq f_0$ and so $A \subset B$ where $B=\{f \leq 1\}$. By construction, $B$ is symmetric and radial. It is also bounded. Then it follows from the weak$^*$ lower semi-continuity of $f$ that $B$ is weak$^*$-compact.\\
Take now $x^* \in [\sigma^{\omega^n}]'_{7\varepsilon}(B)$ and assume as we may that $f(x^*)>\frac12$. By Lemma \ref{elementary} and Lemma \ref{treeCharSzlenk} there exist $T\in {\cal T}_{\omega^n}$ and  $(x^*_s)_{s \in T}$ a weak$^*$-continuous and $\frac{3}{2}\eps$-separated family in $B$ such that $x^*_\emptyset=x^*$. Fix $\nu>0$. First, it follows from the weak$^*$ lower semi-continuity of $f$ and the $f_k$'s that we may assume, by considering a subtree of $T$ belonging to ${\cal T}_{\omega^n}$, that for all $s\in T$, $f(x^*_s)>\frac12$ and for all $s\in T$ and all $k$, $f_k(x^*_s)\ge f_k(x^*)-\nu$. \\
On the other hand, for all $j,k\le m$, $f_j\ge f_0\ge rf_k$. It follows that for all $k\le m$, $f\ge \frac{r+r^2}{2}f_k$ and therefore for all $y^* \in B$, $f_k(y^*)\le \frac{2}{r+r^2}$. Then, we have that
$$\forall s\in T\ \ \forall k\in \{0,...,m\}:\ \ \frac23\le\frac{r+r^2}{2}\le f_k(x^*_s)^{-1}:=\lambda_k^s \le 2.$$
Notice that $\{ \lambda_k^s x^*_s : s \in T\}$ is included in $[\sigma^{\omega^n}]^k_\eps(A) \cup rA$.
Then it follows from Lemma \ref{extract} that for any $k=0,\dots,m-1$, there exists $\lambda_k\ge \frac{r+r^2}{2}$ such that \[\lambda_kx^* \in [\sigma]^{\omega^n}_\eps([\sigma^{\omega^n}]^k_\eps(A) \cup rA)\subset [\sigma^{\omega^n}]'_\eps([\sigma^{\omega^n}]^k_\eps(A) \cup rA)\subset [\sigma^{\omega^n}]^{k+1}_\eps(A) \cup rA.\]
The last inclusion  follows from the fact that $\sigma^{\omega^n}$ is a measure of non-compactness (Lemma \ref{iterate}) and the stability of the associated fragment derivation under finite unions.

Still by Lemma~\ref{extract}, we can find $S\subset T$ such that $S\in {\cal T}_{\omega^n}$ and  $\abs{(\lambda_k^s)^{-1}-\lambda_k^{-1}}<\nu$ for all $s\in S$ and all $k\le m-1$.
Now since $\lambda_k x^* \in B_{k+1}$, we obtain
\[
f_{k+1}(x^*)\le \lambda_k^{-1}\le (\lambda_k^s)^{-1}+\nu=f_k(x^*_s)+\nu
\]
for any $k=0,\dots,m-1$ and any $s\in S$.
We infer that for all $s\in S$
$$f(x^*)\le \frac12 f_0(x^*_s)+\frac{\nu}{2}+\frac{r}{2(m+1)}f_0(x^*)+\frac{r}{2(m+1)} \sum_{k=1}^{m} f_k(x^*)$$
$$\le \frac12 f_0(x^*_s)+\frac{r}{2(m+1)}\sum_{k=0}^{m} f_k(x_s^*)+\frac{r}{2(m+1)}\big(f_0(x^*_s)-f_m(x^*_s)\big)+\frac{\nu (r+1)}2$$
$$=f(x^*_s)-\frac{r}{2(m+1)}\big(\frac{1}{r}-1\big)f_m(x^*_s)+\frac{\nu (r+1)}2\le 1-\frac{1-r}{4(m+1)}+\frac{\nu (r+1)}2.$$
Since $\nu>0$ was arbitrary, we obtain that $f(x^*)\le 1-\frac{1-r}{4(m+1)}$. Applying this last inequality with $r=\frac78$ as we may, we conclude our proof.
\end{proof}

We now state and prove the main result of this section.

\begin{Thm}\label{main} Let $X$ be a separable Banach space and $\sigma$ be the Kuratowski measure of non-compactness on $X^*$. Then, for any $n$ in  ${\Bbb N}$, $\sigma^{\omega^n}$ is convexifiable.
\end{Thm}

\begin{proof} We can proceed by induction. The claim is true for $n=0$. Indeed, it is easily checked that if a weak$^*$-compact subset $A$ of $X^*$ can be covered by finitely many balls of diameter at most $\eps$, then for any $\delta>0$,  $\overline{\rm conv}^*(A)$ can be covered by finitely many balls of diameter $(1+\delta)\eps$.\\
Assume now that $\sigma^{\omega^n}$ is convexifiable. Denote $\kappa_n$ the convexifiability constant of $\sigma^{\omega^n}$. Let $A$ be a weak$^*$-compact subset of $X^*$ such that $\sigma^{\omega^{n+1}}(A) < \varepsilon$. Then $[\sigma^{\omega^n}]^m_\varepsilon(A)=\emptyset$ for some $m \in {\Bbb N}$. Combining Propositions \ref{radial} and \ref{convex} gives that $[\sigma^{\omega^n}]^\omega_{8\kappa_n\varepsilon}(\overline{conv}^*(A))\subset \langle \sigma^{\omega^n}\rangle^\omega_{8\kappa_n\varepsilon}(\overline{conv}^*(A))=\emptyset$. Therefore $\sigma^{\omega^{n+1}}(\overline{conv}^*(A))\le 8\kappa_n\varepsilon$.
\end{proof}

\begin{Remark} Notice that the constant of convexifiability increases in each step of the induction. It follows from our proof that  $\kappa_n\le 2\cdot8^n$. We do not know if this method can be adapted beyond $\omega^\omega$.
\end{Remark}

We can now compare the Szlenk index and the convex Szlenk index.

\begin{Cor}\label{SzCz} Let $X$ be a separable Banach space.

\smallskip
(1) If $K$ is a weak$^*$-compact convex subset of $X^*$ such that $Sz(K) \le \omega^{n+1}$ for some non negative integer $n$, then $Cz(K)=Sz(K)$.

\smallskip
(2) $Cz(X)=Sz(X)$.
\end{Cor}

\begin{proof} (1) Note first that it follows from Proposition \ref{radial}, Theorem \ref{main} and Proposition \ref{convex} that $\sigma^{\omega^n}$ satisfies the assumptions of Proposition \ref{laststep} with $\theta=8\kappa_n$.\\
Assume now that $K$ is a weak$^*$-compact convex subset of $X^*$ such that $Sz(K) \le \omega^{n+1}$. By Lemma \ref{formSz}, it is enough to show that $Cz(K)\le \omega^{n+1}$. It follows from Lemma \ref{elementary} that for any $\eps>0$,  $[\sigma^{\omega^n}]^\omega_\eps(K)=\emptyset$. Pick now $\lambda_n>8\kappa_n$,...,$\lambda_0>8\kappa_0$. By our initial remark, we obtain that for any $\eps>0$, $\langle \sigma^{\omega^n}\rangle^\omega_{\lambda_n\eps}(K)=\emptyset$. Then applying Proposition \ref{laststep} to $\sigma^{\omega^{n-1}},...,\sigma$ successively implies that for any $\eps>0$, $\langle \sigma \rangle^{\omega^{n+1}}_{\alpha \eps}(K)=\emptyset$, with $\alpha=\lambda_0..\lambda_n$. We have proved that $Kz(K)\le \omega^{n+1}$, or equivalently that $Cz(K)\le \omega^{n+1}$.

\smallskip (2) We may assume that $Sz(X)<\infty$ and, by Lemma \ref{formSz} it is enough to show that $Cz(X)= \omega^{\alpha}$ whenever  $Sz(X)= \omega^{\alpha}$, with $\alpha$ countable ordinal.

Assume first that $\alpha\ge \omega$. We need to introduce a new derivation. For a weak$^*$-compact convex subset $K$ of $X^*$ and $\eps>0$, we define $d_\eps'(K)$ to be the set of all $x^*\in K$ such that for any weak$^*$-open halfspace $H$ of $X^*$ containing $x^*$, the diameter of $K\cap \overline{H}^*$ is at least $\eps$. Then $d_\eps^\alpha(K)$ is defined inductively for $\alpha$ ordinal as usual, $Dz(K,\eps)=\inf\{\alpha,\ d_\eps^\alpha(K)=\emptyset\}$ if it exists (and $=\infty$ otherwise) and $Dz(K)=\sup_{\eps>0}Dz(K,\eps)$. Finally $Dz(X):=Dz(B_{X^*})$.\\
It is clear that for any  weak$^*$-compact convex subset $K$ of $X^*$ and any $\eps>0$, $s'_\eps(K)\subset d_\eps'(K)$. Since $d_\eps'(K)$ is weak$^*$-compact and convex, we have that the weak$^*$-closed convex hull of $s'_\eps(K)$ is included in $d_\eps'(K)$. Then an easy induction combined with (\ref{s-sigma}) yields that $Cz(K)\le Dz(K)$.\\
We now need to recall an important recent result of P.~H\'{a}jek and T.~Schlumprecht who proved in \cite{HajekSchlumprecht2014} that if $Sz(X)=\omega^\alpha$ with $\alpha\in [\omega,\omega_1)$, then $Sz(X)=Dz(X)$. Since we always have $Sz(X)\le Cz(X)\le Dz(X)$, it follows that $Cz(X)=Sz(X)$, whenever $X$ is Banach space such that $Sz(X)=\omega^\alpha$ with $\alpha\in [\omega,\omega_1)$.

Assume now that $Sz(X)=\omega^n$, with $n\in \Ndb$. By applying (1) to $K=B_{X^*}$, we get that $Sz(X)=Cz(X)$.
\end{proof}

\begin{Remark} Let us emphasize again the fact that the equality of the Szlenk index and the convex Szlenk index of a Banach space is a direct consequence of the work of H\'{a}jek and Schlumprecht, except for spaces with Szlenk index $\omega^n$ with $n$ finite. The main result of this section fills this gap.

It is worth mentioning that H\'{a}jek and Schlumprecht also proved in \cite{HajekSchlumprecht2014} that if $Sz(X)=\omega^n$ with $n$ finite, then $Dz(X)\le \omega^{n+1}$ and that this result is optimal (see \cite{HajekLancienProchazka2009}).

\end{Remark}

Our comparison extends to the non-separable setting as follows.

\begin{Cor}
Let $X$ be a Banach space such that $Sz(X)<\omega_1$, where $\omega_1$ is the first uncountable ordinal.  Then $Cz(X)=Sz(X)$.
\end{Cor}

\begin{proof} Let $\alpha=Sz(X)<\omega_1$. Then for any separable subspace $Y$ of $X$, $Sz(Y)\le \alpha$. Thus, Corollary \ref{SzCz} ensures that for any separable subspace $Y$ of $X$, $Cz(Y)=Sz(Y)\le \alpha$. Then, since $\alpha$ is countable, it follows from the techniques developed in \cite{Lancien1996} to show the separable determination of such indices (see Propositions 3.1 and 3.2) that $Cz(X)\le \alpha=Sz(X)$, which concludes the proof.
\end{proof}

\section{Renorming spaces and Szlenk index}\label{norm}

The aim of this section is to generalize the following theorem due to Knaust, Odell and Schlumprecht \cite{KnaustOdellSchlumprecht1999} (see \cite{GKL2001} for quantitative improvements  and \cite{Raja2010} for the extension to the non-separable case).

\begin{Thm} Let $X$ be a separable Banach space such that $Sz(X)\le \omega$. Then $X$ admits an equivalent norm, whose dual norm satisfies the following property: for any $\eps>0$ there exists $\delta>0$ such that $[\sigma]'_\eps(B_{X^*}) \subset (1-\delta)B_{X^*}$.
\end{Thm}

Note, that  an easy homogeneity argument shows that the converse of this statement is clearly true. A dual norm satisfying the conclusion of the above theorem is said to be \emph{weak$^*$ uniformly Kadets-Klee} (in short UKK$^*$). We now introduce the following analogous definition.

\begin{Def} Let $X$ be a Banach space and $\alpha \in [0,\omega_1)$ an ordinal. The dual norm on $X^*$ is \emph{$\omega^{\alpha}$-weak$^*$ uniformly Kadets-Klee} (in short $\omega^\alpha$-UKK$^*$) if for any $\eps>0$ there exists $\delta>0$ such that $[\sigma]^{\omega^\alpha}_\eps(B_{X^*}) \subset (1-\delta)B_{X^*}$.
\end{Def}

Our main renorming result is the following.

\begin{Thm}\label{renorming} Let $X$ be a separable Banach space. Then $Sz(X)\le \omega^{\alpha+1}$ if and only if $X$ admits an equivalent norm whose dual norm is $\omega^\alpha$-UKK$^*$.
\end{Thm}

It is clear that a Banach space $X$ with a dual $\omega^\alpha$-UKK$^*$ norm satisfies $Sz(X)\le \omega^{\alpha+1}$. So we shall concentrate on the other implication. Before our proof we need a few technical lemmas and definitions.

\begin{Lem}\label{l:TreeTransport}
Let $0\le \alpha<\omega_1$, $0<2a<b$ and $T \in {\cal T}_\alpha$. Assume that $A \subset B \subset X^*$ are two weak$^*$-compact sets and that $(x^*_s)_{s \in T} \subset B$ is a $b$-separated, weak$^*$-continuous family such that  $\dist(x^*_s,A)< a$ for all $s\in T$. Then there exists $S \in {\cal T}_\alpha$, $S \subset T$ and a weak$^*$-continuous and $(b-2a)$-separated family $(y^*_s)_{s\in S} \subset A$ such that
$\|x^*_\emptyset - y^*_\emptyset\|\leq a$.
\end{Lem}

\begin{proof}
The proof goes by induction on $\alpha$. The claim is clear when $\alpha=0$ so let us assume that we have proved our assertion for every $\beta<\alpha$. Then, there is a sequence $(\alpha_k)$ of  ordinals in $[1,\alpha)$ (with $\alpha_k \nearrow \alpha$ if $\alpha$ is a limit ordinal, $\alpha_k+1=\alpha$ if $\alpha$ is a successor ordinal), and a sequence $(n_k)$ in $\Natural$ such that
\[
T=\set{\emptyset}\cup  \bigcup_{k \in \Natural} \set{(n_k)^\smallfrown s: s \in T_{k}},
\]
where $T_k\in {\cal T}_{\alpha_k}$ for all $k\in \Ndb$. It follows from our induction hypothesis, that for each $k\in \Ndb$ there exist a tree $S_k \in
{\cal T}_{\alpha_k}$, $S_k \subset T_k$, and a weak$^*$-continuous and $(b-2a)$-separated family that we denote $(y^*_{(n_k)\smallfrown s})_{s\in S_k} \subset A$ such that the roots $y^*_{(n_k)}$ of these families
satisfy $\|y^*_{(n_k)}-x^*_{(n_k)}\|\leq a$. By passing to a subsequence, we may assume that the roots $y^*_{(n_k)}$ of these families are such that $y^*_{(n_k)} \wstoo y^*_\emptyset$. Then $\|y^*_\emptyset-x^*_\emptyset\|\leq a$ and $\|y^*_{(n_k)}-y^*_\emptyset\| \geq b-2a$. Finally, the tree
\[
S=\set{\emptyset}\cup \bigcup_{k \in \Natural} \set{(n_k)^\smallfrown s: s \in S_{k}},
\]
belongs to ${\cal T}_\alpha$ and $(y^*_s)_{s\in S}$ satisfies the desired properties.
\end{proof}

We shall now define inductively the class  ${\cal L}_\alpha(T)$ of ``converging'' real valued functions on a given tree $T$ in ${\cal T}_\alpha$ and their ``limit'' along $T$.
\begin{Def} For $T \in {\cal T}_0$ and $r:T \to \Real$ we put $\lim_T r:=\lim_{s \in T} r(s):=r(\emptyset)$. We define ${\cal L}_0(T)=\Real^T$.
Let now $\alpha \in [1,\omega_1)$ and assume that the class ${\cal L}_\beta(T)$ has been defined for all $\beta<\alpha$ and all $T\in {\cal T}_\beta$. Assume also that  for all $T\in {\cal T}_\beta$ and all $r\in {\cal L}_\beta(T)$, $\lim_{s\in T} r_s$ has been defined.
Consider now $T \in {\cal T}_\alpha$ and $r:T \to \Real$.  Then
\[
 T=\set{\emptyset}\cup \bigcup (n_k)^\smallfrown T_k
\]
with $\alpha_k \nearrow \alpha$ if $\alpha$ is a limit ordinal, $\alpha_k+1=\alpha$    if $\alpha$ is a successor ordinal and for all $k\in \Ndb$, $T_k\in {\cal T}_{\alpha_k}$. We say that $r \in {\cal L}_\alpha(T)$ if for all $k\in \Ndb$,  $r\restricted_{T_k} \in {\cal L}_{\alpha_k}(T_k)$ and $\lim_{k \to \infty} \lim_{s \in T_k} r(s)$ exists. Then we set $\lim_T r:=\lim_{s \in T}r(s):=\lim_{k \to \infty} \lim_{s \in T_k} r(s)$
\end{Def}
Observe that the existence and the value of $\lim_T r$ depends only on $r\restricted_{T\setminus T'}$.
The following observations rely on straightforward transfinite inductions similar to the one used in the proof of Lemma \ref{extract}.

\begin{Lem}\label{l:BolzanoWeierstrass} Let $\alpha\in [0,\omega_1)$, $T \in {\cal T}_\alpha$ and $r:T\to \Rdb$.

\smallskip

(i) Assume that $r\in {\cal L}_\alpha(T)$ and that $S\subset T$ with $S\in {\cal T}_\alpha$. Then $r\restricted_S \in {\cal L}_\alpha(S)$ and $\lim_T r=\lim_S r$.

\smallskip

(ii) Assume that $r:T \to \Real$ is bounded. Then there exists $S\subset T$ such that $S\in {\cal T}_\alpha$ and $r\restricted_S \in  {\cal L}_\alpha(S)$.

\smallskip

(iii) Assume that $r\in  {\cal L}_\alpha(T)$. Then for each $\varepsilon>0$ there exists $S\subset T$ such that $S\in {\cal T}_\alpha$ and for all $s \in S \setminus S^1$, we have $|r(s)-\lim_{T}r|<\varepsilon$.
\end{Lem}

We can now proceed with the proof of Theorem \ref{renorming}. We will adapt to this new situation a construction of uniformly convex norms given in \cite{Lancien1995}.

\begin{proof}[Proof of Theorem \ref{renorming}] So let us assume that $Sz(X)\le \omega^{\alpha+1}$. Then we get from Corollary \ref{SzCz} that $Cz(X)\le \omega^{\alpha+1}$. Fix $k\in \Ndb$. We define inductively for $n\in \Ndb$:
$$A^0_k:=B_{X^*}, \ \  A^{n+1}_k:=\overline{\rm conv}^*\big([\sigma]^{\omega^\alpha}_{2^{-k}}(A^n_k)\big).$$
The fact that $Cz(X)\le \omega^{\alpha+1}$ implies that for all $k\in \Ndb$, there exists $n\in \Ndb$ such that $A^n_k=\emptyset$. Then denote $N_k:=\min\{n \in \Natural: A^n_k=\emptyset\} -1$. We define
\[f(x^*)=\|x^*\|+\sum_{k=1}^\infty\frac{1}{2^k N_k}\sum_{n=1}^{N_k} \dist(x^*,A^n_k).
\]
It is easily checked that the sets $A^n_k$ are symmetric and convex.
We define $\abs{\ }$ on $X^*$ to be  the Minkowski functional of the set $C=\set{f\leq 1}$. Since $\|x^*\|\le f(x^*)\le 2\|x^*\|$, we have that $|\ |$ is an equivalent norm on $X^*$ satisfying $\|x^*\|\le |x^*| \le 2\|x^*\|$. Moreover, the sets $A^n_k$ are weak$^*$-closed. Therefore $f$ is weak$^*$ lower semi-continuous and $|\ |$ is the dual norm of an equivalent norm on $X$, still denoted $|\ |$.

Let now $\varepsilon>0$ and $x^* \in s_\varepsilon^{\omega^\alpha}(\newball)$ (the distances and diameters are meant with the original norm $\|\ \|$). Then there exist $T \in {\cal T}_{\omega^\alpha}$ and $\displaystyle(x^*_s)_{s \in T} \subset \newball$ weak$^*$-continuous and $\frac{\eps}{2}$-separated such that $x^*_\emptyset=x^*$. For $k\in \Ndb$ and $l\le N_k$, we define $r^l_k:T \to \Real$  by $r^l_k(s):= \dist(x_s^*,A^l_k)$.

Let $k \geq 1$ such that $\frac\varepsilon8 \leq 2^{-k}<\frac\varepsilon4$,  $\xi=\frac{\varepsilon}{64N_k}$ and $k_0 >k$ such that $\sum_{i=k_0}^\infty 2^{-i}<\frac{\xi}{2^kN_k}$. By passing to a subtree, we may assume, using Lemma \ref{l:BolzanoWeierstrass}, that for all $i<k_0$ and all $l\le N_i$, $r^l_i \in {\cal L}_{\omega^\alpha}(T)$. So, for each $i<k_0$ and $1\leq l\leq N_i$ we denote $d^l_i:=\displaystyle\lim_{s \in T} \dist(x^*_s,A_i^l)$. Then by using Lemma \ref{l:BolzanoWeierstrass}~(iii) and passing to a further subtree, we may assume that for each $i<k_0$ and each $1\leq l\leq N_i$ we have $\abs{d(x^*_s,A_i^l)- d^l_i} < \frac{\xi}{2^kN_k}$ for all $s \in T\setminus T^1$.
By the weak$^*$ lower semi-continuity of the distance functions this implies that $d(x_s^*, A_i^l)\leq d^l_i +\frac{\xi}{2^kN_k}$ for all $s \in T$. Note also that we have $d(x^*, A_i^l) \leq d^l_i$. Let now $\gamma=\frac{\varepsilon}{16N_k}=4\xi$.

\begin{Claim} \label{gap} There exists $l \in \set{1,\ldots,N_k}$ such that $\dist(x^*,A_k^l)\le d^l_k - \gamma$.
\end{Claim}
\begin{proof}[Proof of Claim \ref{gap}] Otherwise for all $l \in \set{1,\ldots,N_k}$ we have $\dist(x^*,A_k^l)> d^l_k-\gamma$.
Then we will show by induction that for all $l\le N_k$, $d^l_k<\gamma l+\xi(l-1)$.\\
For $l=1$ we have that $x^* \in s^{\omega^\alpha}_{2^{-k}}(\newball)\subset A_k^1$. Therefore $d^1_k<\gamma$.\\
If $d^l_k<\gamma l+\xi(l-1)$ has been proved for $l\leq N_k-1$, we can use Lemma~\ref{l:TreeTransport} with the following values $A=A_k^l$, $B=B_{|\ |}$, $a=d^l_k+\xi$, $b=\frac\varepsilon2$ and replacing $\alpha$ with $\omega^\alpha$. Notice that $a=d^l_k+\xi<(\gamma+\xi)l\leq\frac{\eps}{8}=\frac14 b$.
So there exists $S\in {\cal T}_{\omega^\alpha}$ and $(y^*_s)_{s\in S}$ in $A_k^l$ which is weak$^*$-continuous, $(\frac\varepsilon2-2\frac\varepsilon8)$-separated so that its root $y^*$  satisfies $\|y^*-x^*\| \leq  d^l_k+\xi < (\gamma+\xi) l$. Note that $y^* \in A_k^{l+1}$. It follows that $d^{l+1}_k-\gamma<\dist(x^*,A_k^{l+1})< (\gamma+\xi) l$ and therefore  $d^{l+1}_k<\gamma (l+1)+\xi l$.\\
Now applying Lemma \ref{l:TreeTransport} with $A=A_k^{N_k}$, $B=B_{|\ |}$, $a=N_k (\gamma+\xi)\leq\frac\varepsilon8$, $b=\frac\varepsilon2$ and replacing $\alpha$ with $\omega^\alpha$, we produce a weak$^*$-continuous $(\frac\varepsilon2-2\frac\varepsilon8)$-separated tree of height $\omega^\alpha$ in $A_k^{N_k}$ which implies that $[\sigma]^{\omega^\alpha}_{2^{-k}}(A_k^{N_k}) \neq \emptyset$. This contradiction proves our claim.
\end{proof}

We now conclude the proof of Theorem \ref{renorming}. Take any $s \in T\setminus T^1$. We recall that  $\dist(x^*_s,A_j^l)>d^l_j-\frac{\xi}{2^kN_k}$ for all $1\leq j < k_0$ and $1\leq l\leq N_j$. With another application of Lemma \ref{l:BolzanoWeierstrass}, we may also assume that for all $s \in T\setminus T^1$ we have that $\|x^*_s\|\geq \lim_{t\in T} \|x^*_t\| - \frac{\xi}{2^kN_k}$. We now have by our choice of $k_0$, $\xi$ and $\gamma$, by Claim \ref{gap} and the weak$^*$ lower semi-continuity of all involved terms, that for all $s \in T\setminus T^1$
\[
 \begin{split}
  f(x^*)&=\|x^*\|+\sum_{j=1}^{k_0-1}\frac{1}{2^jN_j} \sum_{n=1}^{N_j} \dist(x^*, A_j^n)+\sum_{j=k_0}^{\infty}\frac{1}{2^jN_j} \sum_{n=1}^{N_j} \dist(x^*, A_j^n)\\
&\leq \lim_{t \in T_{\omega^\alpha}} \|x^*_t\|+\left(-\frac{\gamma}{2^kN_k}+\sum_{j=1}^{k_0-1}\frac{1}{2^jN_j} \sum_{n=1}^{N_j} d^n_j \right)+\frac{\xi}{2^kN_k}\\
&\leq \|x^*_s\|+\frac{\xi}{2^kN_k}+\left(-\frac{\gamma}{2^kN_k}+\sum_{j=1}^{k_0-1}\frac{1}{2^jN_j} \sum_{n=1}^{N_j} (\dist(x^*_s,A_j^n)+\frac{\xi}{2^kN_k}) \right)+\frac{\xi}{2^kN_k}\\
&\leq f(x^*_s)-\frac{\xi}{2^kN_k}\leq 1-\frac{\varepsilon}{2^k64N_k^2}\le 1-\frac{\eps^2}{2^9N_k^2}.
 \end{split}
\]
We have shown that for any $\eps>0$ there exists $\delta(\eps)>0$ such that $f(x^*)\le 1- \delta(\eps)$, whenever $x^* \in s_\varepsilon^{\omega^\alpha}(B_{|\ |})$. Note now that $f$ is 2-Lipschitz for $\|\ \|_{X^*}$ and that $B_{|\ |} \subset B_{X^*}$. So, for any $x^*\in s_\varepsilon^{\omega^\alpha}(B_{|\ |})$ and any $t\in [1,1+\frac{\delta(\eps)}{2}]$, we have that
$$f(tx^*)\le f(x^*)+2(t-1)\le 1.$$
Therefore
$$\forall x^*\in s_\varepsilon^{\omega^\alpha}(B_{|\ |})\ \ |x^*|\le \frac{1}{1+\frac{\delta(\eps)}{2}}.$$
This proves that $|\ |$ is a $\omega^\alpha$-UKK$^*$ norm.
\end{proof}

\begin{Remark} As we have already mentioned, the Szlenk index of a separable Banach space $X$ is either $\infty$ (exactly when $X^*$ is non-separable) or of the form $Sz(X)=\omega^\alpha$ with $\alpha<\omega_1$. If $\alpha<\omega_1$ is a successor ordinal, it is known since \cite{Samuel1983} that there exists a countable compact metric space $K$ such that $Sz(C(K))=\omega^\alpha$. Let us now mention the complete description of the possible values of the Szlenk index recently obtained by R. Causey (\cite{Causey2016}, Theorem 1.4). The set of all $\alpha<\omega_1$ such that there exists a separable Banach space $X$ satisfying $Sz(X)=\omega^\alpha$ is exactly:
$$[0,\omega_1]\setminus\{\omega^\xi,\ \xi<\omega_1\ {\rm and}\ \xi\ {\rm is\ a\ limit\ ordinal}\}.$$
The case $Sz(X)=\omega^\alpha$ with $\alpha<\omega_1$ limit ordinal is not covered by our renorming theorem. There is a good reason for this. Indeed, in that case, for any $\beta<\omega^\alpha$, $\beta.\omega<\omega^\alpha$. Then the usual homogeneity argument makes it impossible to have $s_\eps^\beta(B_{X^*})\subset (1-\delta)B_{X^*}$  for some $\delta>0$.
\end{Remark}

\medskip\noindent{\bf Aknowledgements.} The authors are extremely grateful to the anonymous referee for his or her numerous valuable comments and corrections which lead to a considerable improvement of the presentation.

\end{document}